\documentclass{amsart}

\usepackage{amssymb}
\usepackage{amsfonts}
\usepackage{amsmath}
\usepackage{amsthm}
\usepackage{graphicx}
\usepackage{mathrsfs}
\usepackage{dsfont}
\usepackage{amscd}
\usepackage{multirow}
\usepackage{palatino}
\usepackage{mathpazo}
\usepackage{mathtools}
\usepackage[all]{xy}
\usepackage[scaled]{helvet} 
\usepackage{courier} 
\normalfont
\usepackage[T1]{fontenc}
\usepackage{calligra}
\usepackage{verbatim}
\usepackage[usenames,dvipsnames]{color}
\usepackage[colorlinks=true,linkcolor=Blue,citecolor=Violet]{hyperref}
\usepackage{cite}
\usepackage{enumitem}

\makeatletter
\@namedef{subjclassname@2010}{%
  \textup{2010} Mathematics Subject Classification}
\makeatother

\theoremstyle{plain}
\newtheorem{theorem}{Theorem}[section]

\newtheorem{lemma}[theorem]{Lemma}
\newtheorem{proposition}[theorem]{Proposition}

\newtheorem{theorem-definition}[theorem]{Theorem-Definition}

\theoremstyle{definition}
\newtheorem{definition}[theorem]{Definition}

\newtheorem{notation}[theorem]{Notation}

\newtheorem{convention}[theorem]{Convention}

\theoremstyle{remark}

\numberwithin{equation}{section}
 
\newcommand{\N}{{\mathds{N}}}

\newcommand{\Q}{{\mathds{Q}}}
\newcommand{\R}{{\mathds{R}}}
\newcommand{\C}{{\mathds{C}}}

\newcommand{\U}{{\mathcal{U}}}
\newcommand{\D}{{\mathfrak{D}}}
\newcommand{\A}{{\mathfrak{A}}}
\newcommand{\B}{{\mathfrak{B}}}
\newcommand{\M}{{\mathfrak{M}}}

\newcommand{\bigslant}[2]{{\raisebox{.2em}{$#1$}\left/\raisebox{-.2em}{$#2$}\right.}}

\newcommand{\Lip}{{\mathsf{L}}}

\newcommand{\TLip}{{\mathsf{T}}}

\newcommand{\dpropinquity}[1]{{\mathsf{\Lambda}^\ast_{#1}}}

\newcommand{\Kantorovich}[1]{{\mathsf{mk}_{#1}}}

\newcommand{\Haus}[1]{{\mathsf{Haus}_{#1}}}

\newcommand{\StateSpace}{{\mathscr{S}}}

\newcommand{\qcms}{quantum compact metric space}

\newcommand{\unit}{1}

\newcommand{\sa}[1]{{\mathfrak{sa}\left({#1}\right)}}

\newcommand{\dom}[1]{{\operatorname*{dom}({#1})}}

\newcommand{\tunnelset}[3]{{\text{\calligra Tunnels}\,\left[\left({#1}\right)\stackrel{#3}{\longrightarrow}\left({#2}\right)\right]}}

\newcommand{\tunnelextent}[1]{{\chi\left({#1}\right)}}

\newcommand{\alg}[1]{{\mathfrak{#1}}}

\newcommand{\BaireSpace}{{\mathscr{N}}}

\newcommand{\norm}[2]{\left\|{#1}\right\|_{#2}}

\newcommand{\diag}{\mathrm{diag}}

\renewcommand{\geq}{\geqslant}
\renewcommand{\leq}{\leqslant}

\allowdisplaybreaks[1]

\makeatletter
\newcommand{\vast}{\bBigg@{4}}
\newcommand{\Vast}{\bBigg@{5}}
\makeatother

\begin{document}

\title[strongly Leibniz Gromov--Hausdorff propinquity]{The strongly Leibniz property and the   Gromov--Hausdorff propinquity}
\author{Konrad Aguilar}
\address{Department of Mathematics and Statistics, Pomona College, 610 N. College Ave., Claremont, CA 91711} 
\email{konrad.aguilar@pomona.edu}
\urladdr{\url{https://aguilar.sites.pomona.edu}}
\thanks{The first author gratefully acknowledges the financial support  from the Independent Research Fund Denmark  through the project `Classical and Quantum Distances' (grant no.~9040-00107B)}

\author[S.R.~Garcia]{Stephan Ramon Garcia}
\address{Department of Mathematics and Statistics, Pomona College, 610 N. College Ave., Claremont, CA 91711} 
\email{stephan.garcia@pomona.edu}
\urladdr{\url{https://pages.pomona.edu/~sg064747}}
\thanks{The second author is partially supported by NSF grants DMS-1800123 and DMS-2054002}

\author{Elena Kim}
\address{Department of Mathematics, Massachusetts Institute of Technology, 77 Massachusetts Ave, Cambridge, MA 02139}
\email{elenakim@mit.edu}
\urladdr{}
\thanks{The third author is supported by the National Science Foundation Graduate Research Fellowship under Grant No. 1745302}

\author{Fr\'ed\'eric Latr\'emoli\`ere}
\address{Department of Mathematics, University of Denver, Denver, CO 80208}
\email{frederic@math.du.edu}
\urladdr{\url{https://math.du.edu/~frederic/}}

\date{\today}

\begin{abstract}
  We construct a new version of the dual Gromov--Hausdorff propinquity that is sensitive to the strongly Leibniz property. In particular, this new distance is   complete on the class of strongly Leibniz quantum compact  metric spaces. Then, given an inductive limit of C*-algebras for which each C*-algebra of the inductive limit is equipped with a strongly Leibniz L-seminorm, we provide sufficient conditions for placing a strongly Leibniz L-seminorm on an inductive limit  such that the inductive sequence converges to the inductive limit in this new Gromov--Hausdorff propinquity. As an application, we place new strongly Leibniz L-seminorms on AF-algebras using Frobenius--Rieffel norms, for which we have convergence of the Effros--Shen algebras in the Gromov--Hausdorff propinquity with respect to their irrational parameter.
\end{abstract}
\maketitle

\section{Introduction and Background}

What is the analogue of a Lipschitz seminorm in noncommutative geometry? Interest in this question stems from the observation that the Lipschitz seminorm defined by the metric of a compact metric space encodes the underlying metric at the level of its C*-algebra of $\C$-valued continuous functions. Thus, an appropriate notion of a noncommutative Lipschitz seminorm is the core ingredient in a noncommutative metric space. A good definition of a noncommutative Lipschitz seminorm should be flexible enough to include many interesting examples and strong enough to enable the development of an interesting theory. For our purposes, an interesting theory is a theory of convergence which extends the Gromov--Hausdorff distance \cite{Gromov} between compact metric spaces to noncommutative geometry.

The origin of the theory of {\qcms s} is found in Connes' work on spectral triples, which provide noncommutative generalizations of Riemannian manifolds \cite{Connes,Connes89}. Rieffel then addressed the question of defining noncommutative analogues of metric spaces without reference to a differential structure \cite{Rieffel98a,Rieffel99}, designed to be   flexible  while still enabling him to define the quantum Gromov--Hausdorff distance between such spaces \cite{Rieffel00}. Quantum compact metric spaces were  given by ordered pairs $(\A,\Lip)$ of an order unit space $\A$ and a seminorm $\Lip$ that induces a metric on the state space $\StateSpace(\A)$ of $\A$ that metrizes the weak* topology, called the Monge--Kantoroich metric, denoted $\Kantorovich{\Lip}$. The metric  $\Kantorovich{\Lip}$ generalizes the Monge--Kantorovich distance \cite{Kantorovich40}, built from the usual Lipschitz seminorm, and Connes distance \cite{Connes89}, built from spectral triples.

The study of Rieffel's quantum Gromov--Hausdorff distance leads  to  two natural questions. One concerns the coincidence property. If $\A$ and $\B$ are two unital C*-algebras such that $(\mathfrak{sa}(\A),\Lip_\A)$ and $(\mathfrak{sa}(\B),\Lip_\B)$, where $\mathfrak{sa}(\A)$ is the order unit space of self-adjoint elements of $\A$, are Rieffel's {\qcms s} which are  at distance zero with respect to Rieffel's distance, then it is not clear that $\A$ and $\B$ are   *-isomorphic. Attempts at strengthening Rieffel's construction to get this desirable property typically involve restricting the notion of a quantum compact metric spaces, for example, to a class of operator systems \cite{Kerr02} or to C*-algebras \cite{Latremoliere13}. The other question is whether one can extend Rieffel's distance to encompass more structures than the quantum metrics. For example, when working with {\qcms s} of the form $(\mathfrak{sa}(\A),\Lip)$ with $\A$ a unital C*-algebra, we are interested in convergence for modules, group actions, and other such higher structures over the underlying C*-algebras. The two matters are related: they   require   an analogue of the Gromov--Hausdorff distance which is well adapted to C*-algebras. For instance, Rieffel's work on convergence of modules \cite{Rieffel06,Rieffel08,Rieffel10,Rieffel10c,Rieffel15} highlighted that a relationship between L-seminorms and the multiplicative structure of the underlying C*-algebra of a {\qcms} was desirable. Rieffel introduced the notion of a \emph{strongly Leibniz} L-seminorm to this end. However, the quantum Gromov--Hausdorff distance is not well adapted to working with such seminorms (see the discussion of the proximity in \cite{Rieffel10c}).

A possible answer to the search for a noncommutative analogue of the Gromov--Hausdorff distance adapted to C*-algebras and Leibniz seminorms was proposed by the fourth author, with the introduction of the propinquity \cite{Latremoliere13,Latremoliere13c,Latremoliere14,Latremoliere15} on a class of {\qcms s} constructed out of C*-algebras and L-seminorms satisfying some form of Leibniz inequality. Moreover, the dual Gromov--Hausdorff propinquity \cite{Latremoliere13c} is complete on the class of Leibniz quantum compact metric spaces. However, the {\em strongly Leibniz} property has not been addressed. Thus, in Section \ref{s:s-leibniz-complete}, we introduce a form of the dual Gromov--Hausdorff propinquity on the class of {\em strongly Leibniz} quantum compact metric spaces and show that it is complete on this class. In Section \ref{s:af}, we apply this to construct strongly Leibniz L-seminorms on certain  inductive limits of C*-algebras for which the inductive sequence converges to the inductive limit in the Gromov--Hausdorff propinquity. To do this, we assume there exist strongly Leibniz L-seminorms on the C*-algebras of the inductive sequence that satisfy  natural assumptions,  as done in \cite{Aguilar18}. Moreover, we find strongly Leibniz L-seminorms on all unital AF-algebras equipped with faithful tracial state using Frobenius--Rieffel norms \cite{Rieffel74, Aguilar-Garcia-Kim21} following a suggestion from Rieffel. These new seminorms  still preserve the convergence results of \cite{Aguilar-Latremoliere15}, including the convergence of Effros--Shen algebras with respect to their irrational parameters.

We begin with the following notion of {\qcms s}.

\begin{notation}
  We denote the norm of a normed vector space $E$ by $\norm{\cdot}{E}$.
\end{notation}

\begin{definition}\label{qcms-def} For $(A,B)\in [1,\infty)\times [0,\infty)$, 
  an {\em {$(A,B)$-\qcms} }$(\A,\Lip)$    is   a unital C*-algebra $\A$ and a seminorm $\Lip$ defined on a dense Jordan-Lie subalgebra $\dom{\Lip}$ of $\mathfrak{sa}(\A)$ that satisfies the following.
  \begin{enumerate}
  \item $\{ a \in \dom{\Lip} : \Lip(a) = 0 \} = \R \unit_\A$.
  \item The Monge-Kantorovich distance $\Kantorovich{\Lip}$ defined on the state space $\StateSpace(\A)$ of $\A$ by:
    \begin{equation*}
      \forall\varphi,\psi \in \StateSpace(\A) \quad \Kantorovich{\Lip}(\varphi,\psi) = \sup\left\{ \left|\varphi(a) - \psi(a)\right| : a\in\dom{\Lip},\Lip(a) \leq 1 \right\}
    \end{equation*}
    metrizes the weak* topology of $\StateSpace(\A)$.
  \item For all $a,b \in \dom{\Lip}$, 
    \begin{equation*}
      \max\left\{ \Lip\left(\frac{ab+ba}{2}\right), \Lip\left(\frac{ab-ba}{2i}\right) \right\} \leq A\left(\Lip(a)\norm{b}{\A} + \norm{a}{\A}\Lip(b)\right) + B \Lip(a) \Lip(b) \text.
    \end{equation*}
  \item $\{ a \in \dom{\Lip} : \Lip(a) \leq 1 \}$ is closed in $\A$.
  \end{enumerate}
  We call $\Lip$ an $(A,B)$-L-seminorm or  an L-seminorm when context is clear.
\end{definition}

To ease our notation,   we adopt the following convention.

\begin{convention}\label{std-convention}
  We fix a class of $(A,B)$-quasi-Leibniz {\qcms s} for some $A\geq 1$ and $B\geq 0$.  All {\qcms s} belong to this class without further mention.
\end{convention}

The class of {\qcms s} can be turned into the objects of a category \cite{Latremoliere16}; for our purpose, quantum isometries, defined below, will provide us with an adequate notion of morphisms.

\begin{definition}
  A \emph{quantum isometry} $\pi : (\A,\Lip_\A)\rightarrow (\B,\Lip_\B)$ between two {\qcms s} $(\A,\Lip_\A)$ and $(\B,\Lip_\B)$ is a surjective *-morphism $\pi : \A\rightarrow\B$ such that $\pi(\dom{\Lip_\A})\subseteq\dom{\Lip_\B}$ and
  \begin{equation*}
    \forall b \in \dom{\Lip_\B} \quad \Lip_\B(b) = \inf\left\{ \Lip_\A(a) : a\in\dom{\Lip_\A}, \pi(a) = b \right\} \text.
  \end{equation*}

  A \emph{full quantum isometry} $\pi:(\A,\Lip_\A)\rightarrow(\B,\Lip_\B)$ is a quantum isometry such that $\pi$ is a *-isomorphism  and $\pi^{-1}$ is a quantum isometry as well.
\end{definition}

The propinquity is   a metric, up to full quantum isometry, between {\qcms s}, defined as follows.
\begin{definition}
  Let $(\A_1,\Lip_1)$ and $(\A_2,\Lip_2)$ be   {\qcms s}. A \emph{tunnel} $\tau = (\D,\Lip_\D,\pi_1,\pi_2)$ from $(\A_1,\Lip_1)$ to $(\A_2,\Lip_2)$ is  a {\qcms} $(\D,\Lip_\D)$ and quantum isometries $\pi_j : (\D,\Lip_\D) \rightarrow (\A_j,\Lip_j)$ for $j=1,2$.

  The \emph{extent} $\tunnelextent{\tau}$ of the tunnel $\tau$ is defined as
  \begin{equation*}
    \tunnelextent{\tau} = \max_{j\in\{1,2\}} \Haus{\Kantorovich{\Lip_\D}}(\StateSpace(\D),\{ \varphi\circ\pi_j : \varphi\in\StateSpace(\A_j) \} ) \text.
  \end{equation*}
\end{definition}

\begin{notation}
  Let $(\A,\Lip_\A)$ and $(\B,\Lip_\B)$ be two {\qcms s}. The class of all tunnels from $(\A,\Lip_\A)$ to $(\B,\Lip_\B)$ is denoted by
  \begin{equation*}
    \tunnelset{\A,\Lip_\A}{\B,\Lip_\B}{}{} \text.
  \end{equation*}
  We emphasize that, by Convention (\ref{std-convention}), if \[\tau = (\D,\Lip_\D,\ldots) \in \tunnelset{\A,\Lip_\A}{\B,\Lip_\B}{}\] then $(\D,\Lip_\D)$ is an $(A,B)$-{\qcms} where $A$ and $B$ are fixed throughout the construction of the propinquity. 
\end{notation}

\begin{definition}
  The {\em propinquity} between two {\qcms s} $(\A,\Lip_\A)$ and $(\B,\Lip_\B)$ is:
  \begin{equation*}
    \dpropinquity{}((\A,\Lip_\A),(\B,\Lip_\B)) = \inf\left\{ \tunnelextent{\tau} : \tau \in \tunnelset{\A,\Lip_\A}{\B,\Lip_\B}{}\right\} \text.
  \end{equation*}
\end{definition}

By exploiting the Leibniz property, the Gromov--Hausdorff propinquity provides an analogue of the Gromov--Hausdorff distance to noncommutative geometry \cite{Latremoliere13,Latremoliere13c,Latremoliere14,Latremoliere15}, namely, a complete metric over the class of {\qcms s} which is zero between isometrically isomorphic {\qcms s}. Moreover, it induces the same topology as the usual Gromov--Hausdorff distance of the class of compact metric spaces (with the identification between a compact metric space $(X,d)$ and the {\qcms} $(C(X),\Lip)$ with $\Lip$ the usual Lipschitz seminorm).

Definition  \ref{qcms-def}  has proven helpful, and was the foundation for new metrics between certain higher structures over {\qcms s} \cite{Latremoliere16c,Latremoliere18b, Latremoliere18d}, including spectral triples \cite{Latremoliere18e}.

However, Rieffel's work on convergence of modules over {\qcms s} required a strengthening of the Leibniz property by requiring, in addition, that Lip norms be  well behaved with respect to the inverse map \cite{Rieffel09}, as follows.

\begin{definition}[\!{\cite[Definition 2.1]{Rieffel10}}]\label{strongly-Leibniz-def}
  A {\qcms} $(\A,\Lip)$ is \emph{$C$-strongly Leibniz}, for some $C \geq 1$, if for all $a\in \dom{\Lip}\cap \mathrm{GL}(\A)$;
  \begin{enumerate}
  \item $a^{-1} \in \dom{\Lip}$, and 
  \item $\Lip\left(a^{-1}\right) \leq C \norm{a^{-1}}{\A}^2 \Lip(a) \text.$
  \end{enumerate}
  We   say that $\Lip$ is \emph{$C$-strongly Leibniz} when it meets these conditions.
\end{definition}

As seen above, even when  estimating the propinquity between two strongly Leibniz {\qcms s}, the tunnels involved in the computation of the propinquity need not themselves use strongly Leibniz {\qcms s}. However, the propinquity is defined with some flexibility, which enables one to restrict the class of {\qcms s} involved in its construction. Thus, if one wishes to only work with strongly Leibniz {\qcms s}, then we may apply the methods of \cite{Latremoliere13} to obtain such a ``strongly Leibniz propinquity,'' which provides a natural framework, for example, for Rieffel's work on modules. 

\begin{notation}
  Let $C \geq 1$ and  let $(\A,\Lip_\A)$ and $(\B,\Lip_\B)$ be two $C$-strongly Leibniz {\qcms s}. The class of all tunnels from $(\A,\Lip_\A)$ to $(\B,\Lip_\B)$ of the form $(\D,\Lip_\D,\pi_\A,\pi_\B)$, where $(\D,\Lip_\D)$  is also $C$-strongly Leibniz, is denoted  
  \begin{equation*}
    \tunnelset{\A,\Lip_\A}{\B,\Lip_\B}{SL_C}{} \text.
  \end{equation*}
\end{notation}

\begin{definition}
  The \emph{$C$-strongly Leibniz propinquity} $\dpropinquity{SL_C}$ on the class $\mathcal{SL}_C$ of all $C$-strongly Leibniz {\qcms s} is defined by setting 
  \begin{equation*}
    \dpropinquity{SL_C}\left((\A_1,\Lip_1),(\A_2,\Lip_2)\right) = \inf\left\{ \tunnelextent{\tau} : \tau \in \tunnelset{\A_1,\Lip_1}{\A_2,\Lip_2}{SL_C} \right\} 
  \end{equation*}
  for any two $(\A_1,\Lip_1)$, $(\A_2,\Lip_2)$ in $\mathcal{SL}_C$. 
  If $C=1$, we denote $\dpropinquity{SL_C}$ by $\dpropinquity{SL}$.
\end{definition}

Of course, for any  $(\A,\Lip_\A)$, $(\B,\Lip_\B)$ in $\mathcal{SL}_C$, the following inequality holds:
\begin{equation}\label{sl-prop-prop-ineq}
  \dpropinquity{SL_C}((\A,\Lip_\A),(\B,\Lip_\B)) \geq \dpropinquity{}((\A,\Lip_\A),(\B,\Lip_\B)) \text.
\end{equation}
Thus, it is clear that the strongly Leibniz property is a symmetric function which is zero exactly between fully quantum isometric {\qcms s}. We prove in this paper that the strongly Leibniz propinquity is indeed a \emph{complete metric} up to full quantum isometry.

We then provide an application of the completeness of the strongly Leibniz propinquity to inductive limits of strongly Leibniz {\qcms s}. As discussed in \cite{Aguilar18}, a natural question in noncommutative metric geometry is to relate the   important categorical notion of the limit of inductive sequences of C*-algebras with the notion of convergence for the propinquity. The first work in this direction, found in \cite{Aguilar-Latremoliere15}, constructed a $(2,0)$-Leibniz Lip-norm on unital AF algebras with a faithful tracial state: starting from a unital AF algebra $\A$ with some choice of a faithful trace $t$  and some sequence $(\A_n)_{n\in\N}$ of finite-dimensional C*-subalgebras of $\A$ such that $\A$ is the closure of $\bigcup_{n\in\N}\A_n$, we use the existence of a unique surjective $t$-preserving conditional expectation $\mathds{E}_n : \A \rightarrow\A_n$ for each $n\in\N$ to define Lip-norms by
\begin{equation*}
  \forall a\in\sa{\A} \quad \Lip(a) = \sup\left\{ \dim{\A_n}\norm{a-\mathds{E}_n(a)}{\A} : n \in \N \right\},
\end{equation*}
allowing for the value $\infty$. In \cite{Aguilar-Latremoliere15}, the authors  prove that \[\lim_{n\rightarrow\infty}\dpropinquity{}((\A_n,\Lip),(\A,\Lip)) = 0,\] and then establish   continuity results for the propinquity of the classes of Effros--Shen algebras and UHF algebras, naturally parametrized by the Baire space.

However, it is natural to try  to construct Lip-norms on inductive limits of {\qcms s} (not necessarily finite dimensional), under appropriate conditions, without starting with an L-seminorm on the limit, but rather, by exploiting the completeness of the propinquity. We refer to \cite{Aguilar18} for examples. In this paper, we use the completeness of the strongly Leibniz propinquity to construct strongly Leibniz L-seminorms on certain inductive limits of strongly Leibniz {\qcms s}  and apply these results to AF algebras, thus obtaining new quantum metrics on some AF algebras. These new quantum metrics inherit their strong Leibniz properties from the work of Rieffel on the strongly Leibniz property of seminorms built from the standard deviations \cite{Rieffel12}, which we call Frobenius--Rieffel seminorms, and their careful study in finite dimensions in \cite{Aguilar-Garcia-Kim21}. We obtain continuity results for the Effros--Shen algebras and the UHF algebras, as parametrized by the Baire space, in the spirit of \cite{Aguilar-Latremoliere15}, but with our new, strongly Leibniz L-seminorms.

\section{The Strongly Leibniz Gromov--Hausdorff Propinquity}\label{s:s-leibniz-complete}


The dual Gromov--Hausdorff  propinquity \cite{Latremoliere13b,Latremoliere14,Latremoliere15} is defined in a flexible manner and permits one to restrict the choice of tunnels in the construction in order to work in a specific class of {\qcms s}. We  apply this flexibility   to define a specialization of the propinquity to the class of strongly Leibniz {\qcms s} \cite{Rieffel10c}, whose definition we recalled in Definition  \ref{strongly-Leibniz-def}. As it dominates the propinquity, it has the desired coincidence property, and it is obviously symmetric. We begin by proving that it satisfies the triangle inequality.

\begin{lemma}\label{l:sl-tunnel}
  Let  $(\A,\Lip_\A)$, $(\B,\Lip_\B)$ and $(\D,\Lip_\D)$ be in $\mathcal{SL_C}$, and let   $\pi_\A : (\A,\Lip_\A) \twoheadrightarrow (\D,\Lip_\D)$ and $\pi_\B : (\B,\Lip_\B) \twoheadrightarrow (\D,\Lip_\D)$ be quantum isometries. For all $\varepsilon > 0$, if we set, for all $(a,b) \in \dom{\Lip_\A}\oplus\dom{\Lip_\B}$,
  \begin{equation*}
    \TLip(a,b) = \max\left\{ \Lip_\A(a), \Lip_\B(b), \frac{1}{\varepsilon}\norm{\pi_\A(a) - \pi_\B(b)}{\D} \right\},
  \end{equation*}
  then $\TLip$ is a $C$-strongly Leibniz L-seminorm on $\A\oplus\B$.
\end{lemma}

\begin{proof}
  By \cite[Theorem 3.1]{Latremoliere14}, $(\A\oplus\B,\TLip)$ is a {\qcms}. It is thus sufficient to prove that $\TLip$ is $C$-strongly Leibniz.

  Let $(a,b) \in \dom{\Lip_\A}\oplus\dom{\Lip_\B} = \dom{\TLip}$ such that $(a,b) \in \mathrm{GL}(\A\oplus\B)$. Thus $a\in\dom{\Lip_\A}\cap\mathrm{GL}(\A)$ and $b \in \dom{\Lip_\B}\cap \mathrm{GL}(\B)$. Since $\Lip_\A$ and $\Lip_\B$ are $C$-strongly Leibniz, we conclude that $(a,b)^{-1} = (a^{-1},b^{-1}) \in \dom{\TLip}$ and
  \begin{equation*}
    \Lip_\A(a^{-1}) \leq C \norm{a^{-1}}{\A}^2 \Lip_\A(a) \quad  \text{ and } \quad \Lip_\B(b^{-1}) \leq C \norm{b^{-1}}{\B}^2 \Lip_\B(b) \text.
  \end{equation*}

  On the other hand,
  \begin{align*}
    \norm{\pi_\A(a^{-1}) - \pi_\B(b^{-1})}{\D}
    &= \norm{\pi_\A(a^{-1})(\pi_\B(b) - \pi_\A(a))\pi_\B(b^{-1}) }{\D} \\
    &\leq \norm{a^{-1}}{\A} \norm{\pi_\A(a) - \pi_\B(b)}{\D} \norm{b^{-1}}{\B} \\
    &\leq \norm{a^{-1}}{\A}\norm{b^{-1}}{\B} \varepsilon \TLip(a,b) \\
    &\leq \varepsilon \norm{(a^{-1},b^{-1})}{\A\oplus\B}^2 \TLip(a,b) \text.
  \end{align*}

  Therefore, $\TLip(a^{-1},b^{-1}) \leq C \norm{(a^{-1},b^{-1})}{\A\oplus\B}^2 \TLip(a,b)$, as needed.
\end{proof}

Thus, as discussed in \cite{Latremoliere13b,Latremoliere14,Latremoliere15}, we deduce the following lemma.
\begin{lemma}
  For all $C \geq 1$,  the $C$-strongly Leibniz propinquity $\dpropinquity{SL_C}$ is a metric on $\mathcal{SL_C}$.
\end{lemma}

Our purpose is to prove that the $C$-strongly Leibniz propinquity is complete as well. Since the propinquity is complete by Equation (\ref{sl-prop-prop-ineq}), we already have  a description of the limit of any Cauchy sequence in $\left(\mathcal{SL}_C,\dpropinquity{SL_C}\right)$  from \cite{Latremoliere13b}.  What remains to be shown is that the limit  for the metric $\dpropinquity{}$  of a Cauchy sequence for the metric $\dpropinquity{SL_C}$ is indeed $C$-strongly Leibniz; moreover, we have to check the tunnels constructed in \cite{Latremoliere13b} are $C$-strongly Leibniz in the current setting.

\medskip

In this section, we fix $C\geq 1$, and we assume that we are given a sequence $(\A_n,\Lip_n)_{n\in\N}$ of $C$-strongly Leibniz {\qcms s} such that  for each $n\in\N$, there exists  a tunnel $(\tau_n)_{n\in\N} = (\D_n,\TLip_n,\pi_n,\rho_n)_{n\in\N}$ from $(\A_n,\Lip_n)$ to $(\A_{n+1},\Lip_{n+1})$ such that $\tunnelextent{\tau_n} \leq \frac{1}{2^n}$  and $\TLip_n$ is $C$-strongly Leibniz.

Following \cite[Section 6]{Latremoliere13b}, we  set 
\begin{equation*}
  \alg{S} = \left\{ (d_n)_{n\in\N} \in \prod_{n\in\N} \D_n : \forall n\in \N  \ \ \ \rho_n(d_n) = \pi_{n+1}(d_{n+1}) \text{ and }\sup_{n\in\N}\norm{d_n}{\D_n} < \infty \right\} \text,
\end{equation*}
and
\begin{equation*}
  \alg{F} = \bigslant{\alg{S}}{\left\{ (d_n)_{n\in\N} \in \alg{S} : \lim_{n\rightarrow\infty} d_n = 0 \right\}} \text.
\end{equation*}
Let $q : \alg{S} \twoheadrightarrow \alg{F}$ be the canonical surjection, which is a *-epimorphism.

We record the following well-known computation for the quotient norm on $\alg{F}$.

\begin{lemma}\label{limsup-norm-lemma}
  For all $a\in \alg{F}$  and for all $(d_n)_{n\in\N} \in \alg{S}$, if $q((d_n)_{n\in\N}) = a$, then
  \begin{equation*}
    \norm{a}{\alg{F}} = \limsup_{n\rightarrow\infty} \norm{d_n}{\D_n} \text.
  \end{equation*}
\end{lemma}

\begin{proof}
  Note that $\norm{a}{\alg{F}} \leq \norm{d}{\alg{S}}$ whenever $q(d) = a$. Thus,   if $N \in \N$  and   $d^N = (\underbracket[1pt]{0,\ldots,0}_{N \text{ times}}, d_{N}, d_{N+1}, \ldots)$, then $q(d^N) = a$, and thus
  \begin{equation*}
    \norm{a}{\alg{F}} \leq \norm{d^N}{\alg{S}} = \sup_{n\geq N} \norm{d_n}{\D_n} \text.
  \end{equation*}
  Therefore, $\norm{a}{\alg{F}} \leq \limsup_{n\rightarrow\infty} \norm{d_n}{\D_n}$.

  Now, let $\varepsilon > 0$. By definition of the norm on the quotient C*-algebra $\alg{F}$, there exists $e = (e_n)_{n\in\N} \in \alg{S}$ such that $q(e)=a$ and $\norm{e}{\alg{S}} - \varepsilon \leq \norm{a}{\alg{F}} \leq \norm{e}{\alg{S}}$.
 Also $\lim_{n\rightarrow\infty} \norm{e_n - d_n}{\D_n} = 0$. Thus, there exists $N \in \N$ such that  $\norm{e_n-d_n}{\D_n} < \varepsilon $ for all $n \geq N$.

  Therefore,
  \begin{equation*}
    \limsup_{n\rightarrow\infty} \norm{d_n}{\D_n} \leq \limsup_{n\rightarrow\infty}\norm{e_n}{\D_n} + \varepsilon \leq \norm{e}{\alg{S}} + \varepsilon \leq \norm{a}{\alg{F}} + 2 \varepsilon\text.
  \end{equation*}
  
  As $\varepsilon > 0$ is arbitrary, we conclude:
  \begin{equation*}
    \limsup_{n\rightarrow\infty} \norm{d_n}{\D_n} \leq \norm{a}{\alg{F}} \text.\qedhere 
  \end{equation*} 
\end{proof}

For any $d = (d_n)_{n\in\N} \in \sa{\alg{S}}$, we let
\begin{equation*}
  \mathsf{S}(d) = \sup_{n\in\N} \TLip_n(d_n),
\end{equation*}
allowing for the value $\infty$. For all $a\in\sa{\alg{F}}$, we define 
\begin{equation*}
  \mathsf{Q}(a) = \inf\left\{ \mathsf{S}(d) : q(d) = a \right\} \text,
\end{equation*}
again, allowing $\infty$. By \cite{Latremoliere13b}, the seminorm $\mathsf{Q}$ is   a Leibniz L-seminorm on $\alg{F}$. The key fact  established in \cite{Latremoliere13b}  is that $(\A_n,\Lip_n)_{n\in\N}$ converges to $(\alg{F},\mathsf{Q})$ for the propinquity. We now show that, under our conditions, $\mathsf{Q}$ is also $C$-strongly Leibniz.

\begin{lemma}\label{l:complete-sl-lip}
  The L-seminorm $\mathsf{Q}$ is $C$-strongly Leibniz.
\end{lemma}

\begin{proof}
  Let $a\in\dom{\mathsf{Q}}\cap\mathrm{GL}(\alg{F})$ and $\varepsilon > 0$. There exists $d = (d_n)_{n\in\N} \in \dom{\mathsf{S}}$ such that $q(d) = a$, and $\mathsf{S}(d) - \varepsilon \leq \mathsf{Q}(a) \leq \mathsf{S}(d)$.

  Let $e = (e_n)_{n\in\N} \in \sa{\alg{S}}$ such that $q(e) = a^{-1}$. By definition, since $q(e d) = q(d e) = \unit_{\alg{F}}$, we conclude that 
  \begin{equation*}
    \lim_{n\rightarrow\infty} \norm{ d_n e_n - \unit_n }{\D_n} = 0\text,
  \end{equation*}
  and thus, there exists $N\in\N$ such that $\norm{d_n e_n - \unit_n}{\D_n} < 1$ for all $n \geq N$. Therefore, $d_n e_n$ is invertible in $\D_n$ for all $n\geq N$.
  Consequently, if $n \geq N$, then $d_n (e_n (d_n e_n)^{-1}) = \unit_n$. Since $d$ and $e$ are self-adjoint,  $\norm{e_n d_n - \unit_n}{\D_n} < 1$, and thus $e_n d_n \in \mathrm{GL}(\D_n)$ for $n\geq N$; moreover $((e_n d_n)^{-1} e_n) d_n = \unit_n$. Thus, for all $n\geq N$,  $d_n \in \mathrm{GL}(\D_n)$.
  
  Now let $h=(h_n)_{n\in\N} \in \alg{S}$ be defined by setting, for all $n\in \N$:
  \begin{equation*}
    h_n = \begin{cases}
      \unit_n \text{ if $n < N$, }\\
      d_n \text{ if $n \geq N$.}
    \end{cases}
  \end{equation*}
  By construction, $h \in \mathrm{GL}(\alg{S})$ and $q(h) = a$. Moreover, $q(h^{-1}) = a^{-1}$ and  
  \begin{equation*}
    h^{-1} = \left( \unit_0,\ldots,\unit_{N-1},d_N^{-1},d_{N+1}^{-1},\ldots\right) \text.
  \end{equation*}

  Since $\TLip_n$ is $C$-strongly Leibniz,   $h_n^{-1} \in \dom{\TLip_n}$ for all $n \in \N$, and
  \begin{equation*}
    \TLip_n(h_n^{-1}) \leq C \norm{h_n^{-1}}{\D_n}^2 \TLip_n(h_n) \text.
  \end{equation*}

  Moreover, $\mathsf{S}(h) \leq \mathsf{S}(d)$ (since $\TLip_n(\unit_n) = 0$), so $\mathsf{S}(h) \leq \mathsf{Q}(a) + \varepsilon$. By Lemma (\ref{limsup-norm-lemma}), $\norm{a^{-1}}{\alg{F}} = \limsup_{n\rightarrow\infty} \norm{h^{-1}_n}{\D_n}$. Thus, there exists $N' \in \N$ such that, if $n\geq N'$, then
  \begin{equation*}
    \sup_{n\geq N'} \norm{d_n^{-1}}{\D_n} - \varepsilon \leq \norm{a^{-1}}{\alg{F}} \leq \sup_{n\geq N'} \norm{d_n^{-1}}{\D_n} \text.
  \end{equation*}
  
  Let $N'' = \max\{N,N'\}$ and define   $g \in \alg{S}$   by $g=(g_n)_{n\in\N}$ with
  \begin{equation*}
    \forall n \in \N \quad g_n = \begin{cases}
      \unit_n \text{ if $n < N''$,}\\
      d_n \text{ if $n \geq N''$.}
    \end{cases}
  \end{equation*}
  Once again, note that $q(g) = a$ and $\mathsf{S}(g) \leq \mathsf{S}(h) \leq \mathsf{Q}(a) + \varepsilon$. Moreover, $g \in \mathrm{GL}(\alg{S})$, with $q(g^{-1})=a^{-1}$.
  
  Thus,  
  \begin{align*}
    \mathsf{Q}(a^{-1})
    &\leq \mathsf{S}(g^{-1}) \\
    &\leq \sup_{n\in\N} \TLip_n(g_n^{-1}) \\
    &\leq \sup_{n\geq N''} \TLip_n(d_n^{-1}) \\
    &\leq C \sup_{n\geq N''} \norm{d_n^{-1}}{\D_n}^2 \sup_{n\geq N''} \TLip_n(d_n) \\
    &\leq C \left( \norm{a^{-1}}{\alg{F}} + \varepsilon \right)^2 \mathsf{S}(g) \\
    &\leq C \left( \norm{a^{-1}}{\alg{F}} + \varepsilon \right)^2 \left( \mathsf{Q}(a) + \varepsilon \right) \text.
  \end{align*}

  Since $\varepsilon > 0$ is arbitrary, 
  \begin{equation*}
    \mathsf{Q}(a^{-1}) \leq C \norm{a^{-1}}{\alg{F}}^2 \mathsf{Q}(a)\text.
  \end{equation*}

  Thus, $\mathsf{Q}$ is strongly Leibniz.
\end{proof}

We   now  deduce the following theorem.

\begin{theorem}
 Let $C\geq 1$. The $C$-strongly Leibniz propinquity on $\mathcal{SL}_C$ is complete.
\end{theorem}

\begin{proof}
  We maintain the notations  introduced  above. We have seen that $(\alg{F},\mathsf{Q})$ is a $C$-strongly Leibniz {\qcms}.

  For each $N\in \N$, let
  \begin{equation*}
    \alg{S}_N = \left\{ (d_n)_{n\geq N} \in \prod_{n\geq N} \D_n : \forall n \geq N \quad \rho_n(d_n) = \pi_{n+1}(d_{n+1}), \sup_{n\geq N}\norm{d_n}{\D_n} < \infty \right\}\text.
  \end{equation*}
  We also let $\Pi_N((d_n)_{n\geq N}) = \pi_N(d_N) \in \A_N$ for all $(d_n)_{n\geq N} \in \alg{S}_N$.
  The C*-algebra $\alg{F}$ is naturally *-isomorphic to
  \begin{equation*}
    \bigslant{\alg{S}_N}{\{(d_n)_{n\geq N}\in \alg{S}_N : \lim_{n\rightarrow\infty} d_n = 0 \}}\text;
  \end{equation*}
  we let $q_N : \alg{S}_N \twoheadrightarrow \alg{F}$ be the associated canonical surjection.

  We also let $\mathsf{S}_N : (d_n)_{n\geq N} \mapsto \sup_{n\geq N} \TLip_n(d_n)$ (allowing the value $\infty$). By \cite{Latremoliere13b}, $(\alg{S}_N,\mathsf{S}_N)$ is a {\qcms}.

  In \cite{Latremoliere13b}, the fourth author proved that
  \begin{equation*}
    \forall a \in \sa{\alg{F}} \quad \mathsf{Q}(a) = \inf\left\{ \mathsf{S}_N(d) : d \in \sa{\alg{S}_N}, q_N(d) = a \right\} 
  \end{equation*}
  and $\Pi_N$ is a quantum isometry to $(\A_N,\Lip_N)$. In other words, $(\alg{S}_N,\mathsf{S}_N,\Pi_N,q_N)$ is a tunnel from $(\A_N,\Lip_N)$ to $(\alg{F},\mathsf{Q})$  with extent at most $\frac{1}{2^N}$.

  It suffices to prove that $\mathsf{S}_N$ is $C$-strongly Leibniz. This is   immediate by definition: if $d=(d_n)_{n\geq N} \in \mathrm{GL}(\alg{S}_N)\cap\dom{\mathsf{S}_N}$, then 
  \begin{equation*}
    \TLip_n(d_n^{-1}) \leq C \norm{d_n^{-1}}{\D_n}^2 \TLip_n(d_n)
  \end{equation*}
  for all $n\geq N$, 
  and thus
  \begin{equation*}
    \mathsf{S}_N(d^{-1}) \leq C \norm{d^{-1}}{\alg{S}_N}^2 \mathsf{S}_N(d) \text.
  \end{equation*}
  This completes our proof.
\end{proof}

\section{Inductive limits of strongly Leibniz quantum compact metric spaces}\label{s:af}

In \cite[Section 2]{Aguilar18},  the first author constructed quantum compact metric spaces on inductive limits such that the given inductive sequence converged to the inductive limit in propinquity. In this section, we specialize these results   to strongly Leibniz   quantum compact metric spaces using the results of the previous section. As a main application of this section and article, we find strongly Leibniz L-seminorms on AF-algebras  that allow for explicit estimates in the strongly Leibniz propinquity and convergence of the Effros--Shen algebras now in the class of strongly Leibniz compact quantum metric spaces.

 \begin{theorem}\label{t:strongly-ind-lim}
 Fix $C \geq 1$. Let $\A=\overline{\cup_{n \in \N} \A_n}^{\|\cdot\|_\A}$ be a unital C*-algebra such that $(\A_n)_{n \in \N}$ is a non-decreasing sequence of unital C*-algebras of $\A$. Assume that $(\A_n, \Lip_n)_{n \in \N}$ is a $C$-strongly Leibniz quantum compact metric space for all $n \in \N$. Let $(\beta(j))_{j \in \N}$ be a summable sequence  in $(0,\infty)$.
 
 If for all $n \in \N$
 \begin{enumerate}
 \item   $\Lip_{n+1}(a) \leq \Lip_n(a)$ for all $a \in \A_n$, and 
 
 \item   for all $a \in \A_{n+1}, \Lip_{n+1}(a)\leq 1$, there exists $b \in \A_n, \Lip_n(b) \leq 1$ such that 
 \[
 \|a-b\|_\A<\beta(n),
 \] 
 \end{enumerate}
 then there exists a $C$-strongly Leibniz seminorm   $\Lip$ on $\A$ such that $(\A, \Lip)$ is a quantum compact metric space where
 \[
  \dpropinquity{SL_C}((\A_n,\Lip_n),(\A,\Lip))\leq 4 \sum_{j=n}^\infty \beta(j)
 \]
 for all $n \in \N$, and thus
 \[
 \lim_{n \to \infty} \dpropinquity{SL_C}((\A_n,\Lip_n),(\A,\Lip))=0.
 \]
  \end{theorem}
  \begin{proof}
  Since the tunnels of (1) from    \cite[Theorem 2.15]{Aguilar18} are $C$-strongly Leibniz by Lemma \ref{l:sl-tunnel}, this result follows immediately from \cite[Theorem 2.15]{Aguilar18} and  Lemma \ref{l:complete-sl-lip}.
  \end{proof}

We turn our attention to the AF setting.

\begin{definition}[\! {\cite[Definition 1.5.9 and Tomiyama's  Theorem 1.5.10]{Brown-Ozawa}}]\label{d:cond-exp}
Let $\A$ be a unital C*-algebra and let $\B\subseteq \A$ be a unital C*-subalgebra. A linear map $E : \A \rightarrow \B$ is a {\em conditional expectation} if $E(b)=b$ for all $b \in \B$, $\|E(a)\|_\A\leq \|a\|_\A$ for all $a \in \A$, and $E(bab')=bE(a)b'$ for all $a\in \A, b,b'\in \B$.

A conditional expectation is {\em faithful} if $E(a^*a)=0$ implies $a=0$.
\end{definition}

\begin{theorem-definition}[\! {\cite[Section 5]{Rieffel12}} and \cite{Rieffel74}]\label{td:fr-norms}
Let $\A$ be a unital C*-algebra and let $\B\subseteq \A$ be a unital C*-subalgebra. If $E: \A \rightarrow \B$ is a faithful conditional expectation, then  
\[
\|a\|_E=\sqrt{\|E(a^*a)\|_\A}
\]
defines a norm on $\A$ called the {\em Frobenius--Rieffel norm associated to $E$}. 
\end{theorem-definition}
A quick application of the C*-identity shows that $\|\cdot\|_E \leq \|\cdot \|_\A$. We now place strongly Leibniz L-seminorms on all unital AF-algebras equipped with a faithful tracial state that allow for explicit approximations from the finite-dimensional C*-subalgebras. The following construction comes from the first and last author's work in \cite{Aguilar-Latremoliere15}, where they used the C*-norms instead of the Frobenius--Rieffel norms. However, the L-seminorms of \cite{Aguilar-Latremoliere15} are only known to be quasi-Leibniz with $A=2$ and $ B=0$, and we do not know if they are strongly Leibniz for any $C \geq 1$. Yet, due to a suggestion of Rieffel (to the first author at the Fall 2016 AMS Western section at University of Denver) to replace the C*-norms with his norms from \cite{Rieffel12}, we now have the following strongly Leibniz L-seminorms that still have the desired convergence results of \cite{Aguilar-Latremoliere15}.

\begin{theorem}\label{t:main-af}
Let $\A=\overline{\cup_{n \in \N} \A_n}^{\|\cdot\|_\A}$ be a unital AF algebra equipped with a faithful tracial state $\tau$ such that $\A_0=\C1_\A$. Set $\U=(\A_n)_{n \in \N}$ and 
let $(\beta(n))_{n \in \N}$ be a summable sequence of positive real numbers. For each $n \in \N$, let 
\[
E_n^\tau : \A \rightarrow \A_n
\]
be the unique $\tau$-preserving faithful conditional expectation. For $n \in \N\setminus \{0\}$, let $\kappa_n >0$ such that 
\[
\kappa_n \|a\|_{\A}\leq \|a\|_{E_n^\tau}
\]
for all $a \in \A_{n+1}$, and set $\kappa=(\kappa_n)_{n \in \N}$. For each $n \in \N \setminus \{0\}$, let 
\[
\Lip_{\U_n,\beta}^{\tau, \kappa}(a)=\max_{m \in \{0,1,\ldots, n-1\}} \frac{\max\left\{\|a-E_m^\tau(a)\|_{E_m^\tau}, \|a^*-E_m^\tau(a^*)\|_{E_m^\tau}\right\}}{\kappa_m\beta(m)}
\]
for all $a \in \A_n$, and let $\Lip_{\U_0,\beta}^{\tau, \kappa}=0$.

Then  $(\A_n, \Lip_{\U_n,\beta}^{\tau, \kappa})$ is a strongly Leibniz quantum compact metric space (with $(A,B,C)=(1,0,1)$) for all $n \in \N$, and there exists a seminorm $\Lip_{\U, \beta}^{\tau, \kappa}$ such that $(\A, \Lip_{\U, \beta}^{\tau, \kappa})$ is a  strongly Leibniz quantum compact metric space (with $(A,B,C)=(1,0,1)$) where
 \[
  \dpropinquity{SL}((\A_n,\Lip_{\U_n,\beta}^{\tau, \kappa}),(\A,\Lip_{\U, \beta}^{\tau, \kappa}))\leq 4 \sum_{j=n}^\infty \beta(j)
 \]
 for all $n \in \N$, and thus
 \[
 \lim_{n \to \infty} \dpropinquity{SL}((\A_n,\Lip_{\U_n,\beta}^{\tau, \kappa}),(\A,\Lip_{\U, \beta}^{\tau, \kappa}))=0.
 \]
\end{theorem}

\begin{proof}
Let $n \in \N$. By construction,   for $a \in \A_n,$ $\Lip_{\U_n,\beta}^{\tau, \kappa}(a)=0$ if and only if $a=\mu 1_{\A_n}$ for some $\mu \in \C$. Hence, since $\A_n$ is finite-dimensional,  $\Lip_{\U_n,\beta}^{\tau, \kappa}$ is an L-seminorm on $\A_n$. Furthermore, $(\A_n,\Lip_{\U_n,\beta}^{\tau, \kappa})$ is a strongly Leibniz compact quantum metric space (with $A=1,B=0, C=1$) by \cite[Proposition 5.4 and Theorem 5.5]{Rieffel12}. 

By construction, for all $n \in \N$ and $a \in \A$, we have that 
\[
\Lip_{\U_{n+1},\beta}^{\tau, \kappa}(a)=\Lip_{\U_{n},\beta}^{\tau, \kappa}(a).
\]  Thus, (1) of Theorem \ref{t:strongly-ind-lim} is satisfied. 

For (2),   let $ a \in \A_{n+1}$ such that $\Lip_{\U_{n+1},\beta}^{\tau, \kappa}(a)\leq 1$. Consider $E_n^\tau(a) \in \A_n$.
We have that $\|a-E_n^\tau(a)\|_{E_{n}}\leq \kappa_n\beta(n)$. Thus, \[\kappa_n \|a-E_n^\tau(a)\|_\A \leq \|a-E_n^\tau(a)\|_{E_{n}}\leq \kappa_n\beta(n),\]  so $\|a-E_n^\tau(a)\|_\A\leq \beta(n).$ Note that if $n=0$, then $\Lip_{\U_{n},\beta}^{\tau, \kappa}(E_n^\tau(a))=0\leq 1$  and we are done, so for the rest of the proof assume that $n \geq 1$.

We show that $\Lip_{\U_{n},\beta}^{\tau, \kappa}(E_n^\tau(a))\leq 1$. Note that 
\begin{align*}
&\Lip_{\U_{n},\beta}^{\tau, \kappa}(E_n^\tau(a))\\
& \quad =\max_{m \in \{0,1,\ldots, n-1\}} \frac{\max\left\{\|E_n^\tau(a)-E_m^\tau(E_n^\tau(a))\|_{E_m^\tau}, \|E_n^\tau(a^*)-E_m^\tau(E_n^\tau(a^*))\|_{E_m^\tau}\right\}}{\kappa_m\beta(m)}.
\end{align*}
 Let $m \in \{0,1, \ldots, n\}.$ 
 Now    $E_m^\tau\circ E_n^\tau=E_m^\tau$ by  the proof of \cite[Theorem 3.5]{Aguilar-Latremoliere15}, $E_m$ is positive by \cite[Theorem 1.5.10 (Tomiyama)]{Brown-Ozawa}, and $E_m(\A) \subseteq E_n(\A)$. Thus, 
\begin{align*}
& E_m^\tau([E_n^\tau(a)-E_m^\tau(E_n^\tau(a))]^*[E_n^\tau(a)-E_m^\tau(E_n^\tau(a))])\\
 &\quad = E_m^\tau([E_n^\tau(a^*)-E_m^\tau(E_n^\tau(a^*))][E_n^\tau(a)-E_m^\tau(E_n^\tau(a))])\\
& \quad = E_m^\tau(E_n^\tau(a^*)E_n^\tau(a)-E_n^\tau(a^*)E_m^\tau(E_n^\tau(a))- E_m^\tau(E_n^\tau(a^*))E_n^\tau(a)\\
& \quad \quad \quad  +E_m^\tau(E_n^\tau(a^*))E_m^\tau(E_n^\tau(a)))\\
&\quad =  E_m^\tau(E_n^\tau(a^*)E_n^\tau(a))-E_m^\tau(E_n^\tau(a^*)E_m^\tau(E_n^\tau(a)))- E_m^\tau(E_m^\tau(E_n^\tau(a^*))E_n^\tau(a))\\
& \quad \quad \quad +E_m^\tau(E_m^\tau(E_n^\tau(a^*))E_m^\tau(E_n^\tau(a)))\\
&\quad =  E_m^\tau(E_n^\tau(a^*)E_n^\tau(a))-E_m^\tau(E_n^\tau(a^*))E_m^\tau(E_n^\tau(a))- E_m^\tau(E_n^\tau(a^*))E_m^\tau(E_n^\tau(a))\\
& \quad \quad \quad  +E_m^\tau(E_m^\tau(E_n^\tau(a^*))E_m^\tau(E_n^\tau(a)))\\
&\quad =  E_m^\tau(E_n^\tau(a^*)E_n^\tau(a))-E_m^\tau(a^*)E_m^\tau(a)- E_m^\tau(a^*)E_m^\tau(a)\\
& \quad \quad \quad  +E_m^\tau(E_m^\tau(E_n^\tau(a^*))E_m^\tau(E_n^\tau(a)))\\
& \quad =  E_m^\tau(E_n^\tau(a^*)E_n^\tau(a))-2E_m^\tau(a^*)E_m^\tau(a) +E_m^\tau(E_m^\tau(E_n^\tau(a^*))E_m^\tau(E_n^\tau(a)))\\
&\quad  =  E_m^\tau(E_n^\tau(a^*)E_n^\tau(a))-2E_m^\tau(a^*)E_m^\tau(a) + E_m^\tau(a^*)E_m^\tau(a).
\end{align*}

Similarly, 
\begin{align*}
  &E_m^\tau((a-E_m^\tau(a))^*(a-E_m^\tau(a))) \\
  &\quad = E_m^\tau(a^*a)-2E_m^\tau(a^*)E_m^\tau(a) +E_m^\tau(a^*)E_m^\tau(a)\\
& \quad  =E_m^\tau(E_n^\tau(a^*a))-2E_m^\tau(a^*)E_m^\tau(a) +E_m^\tau(a^*)E_m^\tau(a).
\end{align*}

By \cite[Proposition 1.5.7]{Brown-Ozawa},   $E_n^\tau(a^*a)-E^\tau_n(a^*)E^\tau_n(a)\geq 0$  and hence 
\[
E_m^\tau(E_n^\tau(a^*a))-  E_m^\tau (E^\tau_n(a^*)E^\tau_n(a)) \geq 0
\]
since $E_m^\tau$ is a conditional expectation. 
Thus, 
\begin{align*}&E_m^\tau([E_n^\tau(a)-E_m^\tau(E_n^\tau(a))]^*[E_n^\tau(a)-E_m^\tau(E_n^\tau(a))])\\ & \quad \quad \quad \leq E_m^\tau((a-E_m^\tau(a))^*(a-E_m^\tau(a))).\end{align*}  Since $E_m^\tau([E_n^\tau(a)-E_m^\tau(E_n^\tau(a))]^*[E_n^\tau(a)-E_m^\tau(E_n^\tau(a))]) \geq 0$, we gather
\begin{align*}
 \|E_n^\tau(a)-E_m^\tau(E_n^\tau(a))\|_{E_m^\tau}^2 
&   =\|E_m^\tau([E_n^\tau(a)-E_m^\tau(E_n^\tau(a))]^*[E_n^\tau(a)-E_m^\tau(E_n^\tau(a))])\|_{\A}\\
&   \leq  \|E_m^\tau((a-E_m^\tau(a))^*(a-E_m^\tau(a)))\|_{\A}\\
&   = \|a-E_m^\tau(a)\|_{E_m^\tau}^2.
\end{align*}
Therefore, repeating this process with $a^*$ in place of $a$, we conclude that 
\[
\Lip_{\U_{n},\beta}^{\tau, \kappa}(E_n^\tau(a))\leq \Lip_{\U_{n+1},\beta}^{\tau, \kappa}(a)\leq 1.
\]
The proof is complete by Theorem \ref{t:strongly-ind-lim}.
\end{proof}

In \cite{Aguilar-Garcia-Kim21}, some equivalence constants $\kappa_n$ were obtained explicitly on direct sums of matrix algebras. The next results ensure that  equivalence constants obtained in this way translate to the inductive limit. 

\begin{proposition}\label{p:equiv-const}

Let $(\A_n, \alpha_n)_{n \in \N}$ be an inductive sequence of C*-algebras (see \cite[Section 6.1]{Murphy90}) such that:
\begin{enumerate} 
\item $\A_0=\C$ and $\A_n=\bigoplus_{k=1}^{n_n} \M_{d_{n,k}}(\C)$ for all $n \in \N\setminus \{0\}$, where $d_{n,k} \in \N \setminus \{0\}$ for each $n \in \N \setminus \{0\}$ and $k \in \{1, 2, \ldots, n_n\}$;
\item   $\alpha_n :\A_n \rightarrow \A_{n+1}$ is a unital *-monomorphism for all $n \in \N$;
\item   the inductive limit $\A=\underrightarrow{\lim} \ (\A_n, \alpha_n)_{n \in \N}$ is equipped with a faithful tracial state $\tau$.
  \end{enumerate}
  For each $n \in \N$, let $\alpha^{(n+1)}, \alpha^{(n)} : \A_n \rightarrow \A$ be the canonical unital *-monomorphisms satisfying 
  \begin{equation}\label{eq:ind-lim-maps}
  \alpha^{(n+1)}\circ \alpha_n=\alpha^{(n)}.
  \end{equation}
  Note that $\A=\overline{\cup_{n \in \N}\alpha^{(n)}(\A_n)}^{\|\cdot\|_\A}$ and $\alpha^{(n)}(\A_n) \subseteq \alpha^{(n+1)}(\A_{n+1})$ and $\alpha^{(0)}(\A_0)=\C1_\A$ (see \cite[Section 6.1]{Murphy90}). For each $n \in \N$, let \[E^\tau_n :\A \rightarrow \alpha^{(n)}(\A_n)\]
  denote the unique $\tau$-preserving faithful conditional expectation onto $\alpha^{(n)}(\A_n).$ For each $n \in \N$, let
\begin{equation}\label{eq:fd-trace}
\tau_n=\tau \circ \alpha^{(n)},
\end{equation}
which is a faithful tracial state on  $\A_n$, and let \[E^{\tau_{n+1}}_{n+1,n}: \A_{n+1} \rightarrow \alpha_n(\A_n)\]
be the unique $\tau_{n+1}$-preserving faithful conditional expectation onto $\alpha_n(\A_n)$. Let $\kappa_n >0$ such that 
\[
\kappa_n \|a\|_{\A_{n+1}} \leq \|a\|_{E^{\tau_{n+1}}_{n+1,n}}
\]
for all $a \in \A_{n+1}$.

Then, for all $n \in \N$,   
\[
E^\tau_n \circ \alpha^{(n+1)}=\alpha^{(n+1)}\circ E^{\tau_{n+1}}_{n+1,n},
\]
and moreover,
\[
\kappa_n \|a\|_\A \leq \|a\|_{E^\tau_n}
\]
for all $a \in \alpha^{(n+1)}(\A_{n+1})$.
\end{proposition}

\begin{proof}
Let  $n \in \N$ and let $B_n$ denote the set of matrix units for $\A_n$.   By  \cite[Expression (4.1)]{Aguilar-Latremoliere15} for both conditional expectations $E^\tau_n$ and $E^{\tau_{n+1}}_{n+1,n}$, we have  for all $a \in \A_{n+1}$
\begin{align*}
E^\tau_n(\alpha^{(n+1)}(a))& = \sum_{e \in B_n}\frac{\tau(\alpha^{(n+1)}(a) \alpha^{(n)}(e^*))}{\tau(\alpha^{(n)}(e^*) \alpha^{(n)}(e))} \alpha^{(n)}(e)\\
& = \sum_{e \in B_n}\frac{\tau(\alpha^{(n+1)}(a) \alpha^{(n)}(e^*))}{\tau(\alpha^{(n)}(e^*) \alpha^{(n)}(e))} \alpha^{(n+1)}(\alpha_n(e))\\
& = \alpha^{(n+1)}\left( \sum_{e \in B_n}\frac{\tau(\alpha^{(n+1)}(a) \alpha^{(n)}(e^*))}{\tau(\alpha^{(n)}(e^*) \alpha^{(n)}(e))}  \alpha_n(e) \right)\\
& = \alpha^{(n+1)}\left( \sum_{e \in B_n}\frac{\tau(\alpha^{(n+1)}(a) \alpha^{(n+1)}(\alpha_n(e^*)))}{\tau(\alpha^{(n+1)}(\alpha_{n}(e^*) \alpha_{n}(e))}  \alpha_n(e) \right)\\
& = \alpha^{(n+1)}\left( \sum_{e \in B_n}\frac{\tau_{n+1} (a \alpha_n(e^*))}{\tau_{n+1}(\alpha_{n}(e^*) \alpha_{n}(e))} \alpha_n(e)\right)\\
& = \alpha^{(n+1)}(E^{\tau_{n+1}}_{n+1,n}(a))
\end{align*}
and  so
\begin{align*}
E^\tau_n \circ \alpha^{(n+1)}=\alpha^{(n+1)}\circ E^{\tau_{n+1}}_{n+1,n}.
\end{align*}

Let $a \in \alpha^{(n+1)}(\A_{n+1})$. Thus, there exists a unique $a_{n+1} \in \A_{n+1}$ such that $a=\alpha^{(n+1)}(a_{n+1})$. We have 
\begin{align*}
\|a\|_{E^\tau_n}^2&=\left\|E^\tau_n\left(\alpha^{(n+1)}(a_{n+1})^*\alpha^{(n+1)}(a_{n+1})\right)\right\|_\A\\
& = \left\|E^\tau_n\left(\alpha^{(n+1)}(a_{n+1}^*a_{n+1})\right)\right\|_\A\\
& = \left\| \alpha^{(n+1)}\left(E^{\tau_{n+1}}_{n+1,n}(a_{n+1}^*a_{n+1})\right)\right\|_\A\\
& = \left\|  E^{\tau_{n+1}}_{n+1,n}(a_{n+1}^*a_{n+1}) \right\|_{\A_{n+1}}\\
& \geq \kappa_n^2 \|a_{n+1}\|_{\A_{n+1}}^2\\
& = \kappa_n^2 \|a\|_\A^2.
\end{align*}
Therefore, 
\[
 \kappa_n \|a\|_\A\leq \|a\|_{E^\tau_n}
\]
as desired.
\end{proof}

 The convergence of the Effros--Shen algebras  in \cite{Aguilar-Latremoliere15} relied on a continuous field of L-seminorms on the finite-dimensional subalgebras of the inductive sequence and, although the L-seminorms of Theorem \ref{t:main-af} have a similar structure to those of \cite[Theorem 3.5]{Aguilar-Latremoliere15}, we need two important facts to ensure that the L-seminorms of Theorem \ref{t:main-af} also form a continuous field of L-seminorms in an appropriate sense. The first,  Proposition \ref{p:cond-exp-conv},    takes care of the fact that we switch the C*-norm for Frobenius--Rieffel norms  and the   second fact takes care of the continuity of the equivalence constants, which was already proven in \cite[Theorem 5.2]{Aguilar-Garcia-Kim21} for particular equivalence constants that were calculated explicitly.
 
 \begin{proposition}\label{p:cond-exp-conv}
 Let $N \in \N, n_1, n_2, \ldots, n_N \in \N \setminus \{0\}$, and   $\A=\oplus_{k=1}^N \M_{n_k}(\C)$. Let $M \in \N \setminus \{0\}, m_1, m_2, \ldots, m_M \in \N \setminus \{0\}$, and   $\B=\oplus_{k=1}^M \M_{m_k}(\C)$. Assume that there exists a unital *-monomorphism $\alpha: \B \rightarrow \A$.  For each $n \in \N \cup \{\infty\}$, let $\mathbf{v}^n \in (0,1)^N$ and let $\tau_{\mathbf{v}^n}$ be the faithful tracial state defined for all  $a=(a_1, a_2, \ldots, a_N) \in \A$  by 
 \[
 \tau_{\mathbf{v}^n}(a)=\sum_{k=1}^N \frac{v^n_k}{n_k} \mathrm{Tr} (a_k),
 \] 
 where $ \mathrm{Tr}$ denotes the trace of a matrix. 
 Let $E^{\tau_{\mathbf{v}^n}}: \A \rightarrow \alpha(\B)$ denote the unique $\tau_{\mathbf{v}^n}$-preserving faithful conditional expectation onto $\alpha(\B)$.

 If $(\mathbf{v}^n)_{n \in \N}$ converges to $\mathbf{v}^\infty$ coordinate-wise, then $(\|\cdot\|_{E^{\tau_{\mathbf{v}^n}}})_{n \in \N}$ converges to $\|\cdot\|_{E^{\tau_{\mathbf{v}^\infty}}}$ uniformly on any compact set of $(\A, \|\cdot\|_\A)$ and thus converges pointwise on $\A$.
 \end{proposition}
\begin{proof}
Let $B$ be the set of matrix units for $\B$. Fix $ a=(a_1,a_2, \ldots, a_N) \in \A$. By \cite[Expression (4.1)]{Aguilar-Latremoliere15},   for each $n \in \N\cup \{\infty\}$,
\[
E^{\tau_{\mathbf{v}^n}}(a)=\sum_{e \in B} \frac{\tau_{\mathbf{v}^n}(a\alpha(e^*))}{\tau_{\mathbf{v}^n}(\alpha(e)\alpha(e^*))} \alpha(e).
\]
The condition that $(\mathbf{v}^n)_{n \in \N}$ converges to $\mathbf{v}^\infty$ coordinate-wise is equivalent to weak* convergence of $(\tau_{\mathbf{v}^n})_{n \in \N}$ to $\tau_{\mathbf{v}^\infty}$. Thus, by continuity of addition and scalar multiplication, $(E^{\tau_{\mathbf{v}^n}}(a))_{n \in \N}$ converges to $E^{\tau_{\mathbf{v}^\infty}}(a)$ with respect to $\|\cdot\|_\A$. 

Now, we prove uniform convergence on any compact set of $(\A, \|\cdot\|_\A)$. Let $C \subset \A$ be compact  with respect to $\|\cdot\|_\A$. Let $\varepsilon>0$. By compactness, there exist $N \in \N $ and $ a_1,a_2, \ldots, a_N \in C$ such that 
\[
C \subseteq \bigcup_{k=1}^N \{ a \in \A : \|a-a_k\|_\A<\varepsilon/3\}.
\]
By pointwise convergence, choose $N' \in \N$ such that $\|E^{\tau_{\mathbf{v}^n}}(a_k)-E^{\tau_{\mathbf{v}^\infty}}(a_k)\|_\A<\varepsilon/3$ for all $n \geq N'$ and $k \in \{1,2, \ldots, N\}$. Let $n \geq N'$ and let $a \in C$. Then  there exists $k \in \{1,2,\ldots, N\}$ such that $\|a-a_k\|_\A<\varepsilon/3$. Thus,
\begin{align*}
\|E^{\tau_{\mathbf{v}^n}}(a)-E^{\tau_{\mathbf{v}^\infty}}(a)\|_\A & \leq \|E^{\tau_{\mathbf{v}^n}}(a)-E^{\tau_{\mathbf{v}^n}}(a_k)\|_\A+\|E^{\tau_{\mathbf{v}^n}}(a_k)-E^{\tau_{\mathbf{v}^\infty}}(a_k)\|_\A\\
& \quad \quad  +\|E^{\tau_{\mathbf{v}^\infty}}(a_k)-E^{\tau_{\mathbf{v}^\infty}}(a)\|_\A \\
& < \|E^{\tau_{\mathbf{v}^n}}(a-a_k)\|_\A+\frac{\varepsilon}{3}+\|E^{\tau_{\mathbf{v}^\infty}}(a-a_k)\|_\A\\
& \leq \|a-a_k\|_\A+\frac{\varepsilon}{3}+\|a-a_k\|_\A\\
&<\varepsilon,
\end{align*}
where  Definition \ref{d:cond-exp} is used in the penultimate inequality.

Hence, by the reverse triangle inequality and uniform continuity of the square root function,  $(\sqrt{\|E^{\tau_{\mathbf{v}^n}}(\cdot)\|_\A})_{n \in \N}$ converges to  $\sqrt{\|E^{\tau_{\mathbf{v}^\infty}}(\cdot)\|_\A}$ uniformly on any compact set of $(\A, \|\cdot\|_\A)$. As singletons are compact, we have pointwise convergence on $\A$.
\end{proof}

We now focus on the Effros--Shen algebras and begin with their definition found in \cite{Effros80b}. 
 Let $\theta \in \R$ be irrational. There exists a unique sequence of   integers $(r^\theta_n)_{n \in \N }$  with $r^\theta_n>0$ for all $n \in \N\setminus \{0\}$ such that 
 \[
 \theta =\lim_{n \to \infty} r_0^\theta +\cfrac{1}{r^\theta_1 + \cfrac{1}{r^\theta_2 + \cfrac{1}{r^\theta_3 +\cfrac{1}{\ddots+\cfrac{1}{r^\theta_n}}}}}.
 \]
 When $\theta \in (0,1)$, we have that $r^\theta_0=0$. The sequence $(r^\theta_n)_{n \in \N_0}$ is   the {\em continued fraction expansion of $\theta$} \cite{Hardy38}.

 Next, we  define the finite-dimensional  C*-subalgebras of the Effros--Shen algebras.
 For each $n \in \N$, define \[
p_0^\theta=r_0^\theta, \quad p_1^\theta=1 \quad \text{ and } \quad q_0^\theta=1, \quad q_1^\theta=r^\theta_1,
\]
and   set 
\[
p_{n+1}^\theta=r^\theta_{n+1} p_{n}^\theta+p_{n-1}^\theta
\]
and 
\[
q_{n+1}^\theta= r^\theta_{n+1} q_{n}^\theta+q_{n-1}^\theta.
\]
The sequence $\left(p_{n}^\theta/q_{n}^\theta\right)_{n \in \mathbb{N}_0}$ of {\em convergents} $p^\theta_n/q^\theta_n$ converges to $\theta$. In fact, for each $n \in \N$, 
\[
 \frac{p_n^\theta}{q_n^\theta}=r_0^\theta +\cfrac{1}{r^\theta_1 + \cfrac{1}{r^\theta_2 + \cfrac{1}{r^\theta_3 +\cfrac{1}{\ddots+\cfrac{1}{r^\theta_n}}}}}.
\]
 
 We now define the C*-algebras with which we endow  Frobenius--Rieffel norms. Let $\A_{\theta,0}=\mathbb{C}$ and, for each $n \in \mathbb{N}_0$, let
\[
\A_{\theta,n}=\M_{q_n^\theta}(\C) \oplus \M_{q_{n-1}^\theta}(\C).
\]
These form an inductive sequence with the   maps   
 \begin{equation}\label{eq:theta-alg}
\alpha_{\theta,n}:a\oplus b \in \A_{\theta,n} \mapsto   \diag\left(a, \ldots, a,b \right)  \oplus a \in \A_{\theta,n+1},
\end{equation}
where there are $r^\theta_{n+1}$ copies of $a$ on the diagonal in the first summand of $\A_{\theta,n+1}$. This   is a unital *-monomorphism by construction.  For $n=0$,
\[\alpha_{\theta, 0}: \lambda \in  \A_{\theta,0} \mapsto   \diag(\lambda, \ldots, \lambda)  \oplus \lambda\ \in \A_{\theta,1}.
\]
The {\em Effros--Shen algebra associated to $\theta$} is   the inductive limit (see \cite[Section 6.1]{Murphy90})
\[
\A_\theta=\underrightarrow{\lim} \ (\A_{\theta,n}, \alpha_{\theta,n})_{n \in \N}.
\]

There exists a unique faithful tracial state $\tau_\theta$ on $\A_\theta$ such that for each $n \in \N \setminus \{0\}$,  $\tau_{\theta,n}$ (see Expression \eqref{eq:fd-trace}) is defined for each $(a,b) \in \A_{\theta,n}$ by
\[
\tau_{\theta,n}(a,b)=t(\theta,n)\frac{1}{q_n^\theta}\mathrm{Tr}(a)+(1-t(\theta,n))\frac{1}{q_{n-1}^\theta}\mathrm{Tr}(b),
\]
where   \[t(\theta,n)=(-1)^{n-1}q_n^\theta  (\theta q_{n-1}^\theta -p_{n-1}^\theta ) \in (0,1)\]
 (see \cite[Lemma 5.5]{Aguilar-Latremoliere15}).
 
For each $n \in \N\setminus \{0\}$, define  
 \[
 \kappa_{\theta,n}=\sqrt{\frac{\theta q^\theta_n - p^\theta_n}{\left( \theta q^\theta_{n-2}-p^\theta_{n-2}\right)r^\theta_n (r^\theta_n+1)^2}}
 \]
 as in \cite[Theorem 5.2]{Aguilar-Garcia-Kim21} and let $\kappa_\theta=(\kappa_{\theta,n})_{n \in \N}$,
 \[
 \beta_\theta(n)=\frac{1}{\dim(\A_{\theta,n})}=\frac{1}{(q_n^\theta)^2+(q_{n-1}^\theta)^2},
 \]
 \[
 \mathcal{U}^\theta_n=\alpha^{(n)}_\theta(\A_{\theta,n})
 \]
 as in Expression \eqref{eq:ind-lim-maps}. For each $a \in \alpha^{(n)}_{\theta}(\A_{\theta,n})$,   let
 \begin{equation}\label{eq:fd-es-lip}
 \Lip_{\theta,n}(a)=\Lip^{\tau_\theta, \kappa_\theta}_{\mathcal{U}^\theta_n,\beta_\theta}(a)=\max_{m \in \{0,1,\ldots, n-1\}} \frac{\max\left\{\|a-E_m^{\tau_\theta}(a)\|_{E_m^{\tau_\theta}}, \|a^*-E_m^{\tau_\theta}(a^*)\|_{E_m^{\tau_\theta}}\right\}}{\kappa_{\theta,m}\beta_\theta(m)}
 \end{equation}
 as in Theorem \ref{t:main-af}.

  \begin{theorem}\label{t:es-fd-approx}
  If $\theta \in (0,1)\setminus \Q$, then there exists a strongly Leibniz L-seminorm $\Lip_\theta$ on $\A_\theta$ with $A=1,B=0, C=1$ such that 
  \begin{equation}\label{eq:es-fd-approx}
  \dpropinquity{SL} ((\A_{\theta}, \Lip_\theta), (\A_{\theta,n}, \Lip_{\theta,n})) \leq 4 \sum_{j=n}^\infty \beta_\theta (j)
  \end{equation}
  for all $n \in \N$, and thus 
  \[
  \lim_{n \to \infty} \dpropinquity{SL} ((\A_{\theta}, \Lip_\theta), (\A_{\theta,n}, \Lip_{\theta,n}))=0.
  \]
  \end{theorem}
  \begin{proof}
 Let $n \in \N$.  By \cite[Theorem 5.2]{Aguilar-Garcia-Kim21} and Proposition \ref{p:equiv-const}, 
  \[
  \kappa_{\theta,n} \|a\|_{\A_\theta} \leq \|a\|_{E^{\tau_\theta}_n}
  \]
  for all $ a \in \alpha^{(n+1)}_\theta (\A_{\theta,n+1})$. The proof is complete by Theorem \ref{t:main-af}.
  \end{proof}

We now prove our main result about the Effros--Shen algebras.

\begin{theorem}\label{t:main-es}
The map
\[
\theta \in (0,1) \setminus \Q \longmapsto (\A_\theta, \Lip_\theta),
\]
where $\Lip_\theta$ is defined in Expression \eqref{eq:es-fd-approx}, 
is continuous with respect to the usual topology on $(0,1) \setminus \Q $ and the topology induced by $\dpropinquity{SL}$.
\end{theorem}
\begin{proof}
Let $(\theta(n))_{n \in \N}$ be a sequence in $(0,1)\setminus \Q$ that converges to $\theta(\infty) \in (0,1)\setminus \Q$. Let $\varepsilon>0$. By \cite{Hardy38},  $(1/q_m^{\theta})_{m \in \N}$ is square summable for all $\theta \in (0,1) \setminus \Q$. Moreover, if   $\Phi=\phi-1$, where $\phi$ is the golden ratio, then $q_m^{\theta(n)} \geq q_m^\Phi$ for all $m \in \N $ and $ n \in \N \cup \{\infty\}$. Choose $N_1 \in \N$ such that    \[4\sum_{j=N_1}^\infty \frac{1}{(q_j^\Phi)^2+(q_{j-1}^\Phi)^2}< \frac{\varepsilon}{3}.\]
Then  Theorem \ref{t:es-fd-approx} ensures that
\[
\dpropinquity{SL}((\A_{\theta(n)}, \Lip_{\theta(n)}), (\A_{\theta(n), N_1}, \Lip_{\theta(n),N_1})) \leq 4\sum_{j=N_1}^\infty \frac{1}{(q_j^\Phi)^2+(q_{j-1}^\Phi)^2}< \frac{\varepsilon}{3} 
\]
for all $n \in \N\cup\{\infty\}$.

 By \cite[Proposition 5.10]{Aguilar-Latremoliere15}, choose $N_2 \in \N$ such that 
 \[
 q_{N_1}^{\theta(n)}=q_{N_1}^{\theta(\infty)} \quad \text{ and } \quad  q_{N_1-1}^{\theta(n)}=q_{N_1-1}^{\theta(\infty)}
 \]
 for all $n \geq N_2$. Therefore,  $\A_{\theta(n), N_1}=\A_{\theta(\infty), N_1}$ for all $n \geq N_2$. By the proof of \cite[Lemma 5.12]{Aguilar-Latremoliere15} along with Proposition \ref{p:cond-exp-conv} and  \cite[Theorem 5.2]{Aguilar-Garcia-Kim21}, we have for all $a \in \A_{\theta(\infty), N_1}$  that 
 \[
 \lim_{m \to \infty} \Lip_{\theta(N_2+m),N_1}\circ \alpha_{\theta(N_2+m)}^{(N_1)} (a)=\Lip_{\theta(\infty),N_1}\circ \alpha_{\theta(\infty)}^{(N_1)}(a).
 \]
 Thus, by the same proof as \cite[Lemma 5.13]{Aguilar-Latremoliere15} 
 \[
 \lim_{n \to \infty} \dpropinquity{SL} ((\A_{\theta(n), N_1}, \Lip_{\theta(n), N_1}), (\A_{\theta(\infty), N_1}, \Lip_{\theta(\infty), N_1}))=0.
 \] 
 Therefore, we may choose $N_3 \geq N_2$ such that 
 \[
 \dpropinquity{SL} ((\A_{\theta(n), N_1}, \Lip_{\theta(n), N_1}), (\A_{\theta(\infty), N_1}, \Lip_{\theta(\infty), N_1}))<\frac{\varepsilon}{3}
 \]
 for all  $n \geq N_3.$  Hence, if $n \geq N_3$, then 
 \begin{align*}
&  \dpropinquity{SL} ((\A_{\theta(n)}, \Lip_{\theta(n)}),( \A_{\theta(\infty)}, \Lip_{\theta(\infty)}))\\
& \quad \leq  \dpropinquity{SL} ((\A_{\theta(n)}, \Lip_{\theta(n)}),( \A_{\theta(n),N_1}, \Lip_{\theta(n), N_1}))\\
& \quad \quad \quad  +\dpropinquity{SL} ((\A_{\theta(n), N_1}, \Lip_{\theta(n), N_1}),( \A_{\theta(\infty), N_1}, \Lip_{\theta(\infty), N_1}))\\
&\quad \quad \quad  +\dpropinquity{SL} ((\A_{\theta(\infty), N_1}, \Lip_{\theta(\infty),N_1}),( \A_{\theta(\infty)}, \Lip_{\theta(\infty)}))\\
& \quad <\frac{\varepsilon}{3}+\dpropinquity{SL} ((\A_{\theta(n), N_1}, \Lip_{\theta(n), N_1}),( \A_{\theta(\infty), N_1}, \Lip_{\theta(\infty), N_1}))+\frac{\varepsilon}{3}\\
& \quad <\frac{\varepsilon}{3}+\frac{\varepsilon}{3}+\frac{\varepsilon}{3} =\varepsilon
 \end{align*}
 by the triangle inequality.
 \end{proof}
 
 Thus, we see that the equivalence constants   found in \cite{Aguilar-Garcia-Kim21} were vital in this continuity result. There is nothing that guarantees that any equivalence constant would provide the same result. However,  Proposition \ref{p:sharp-cont} shows that we can also obtain continuity of the map in Theorem \ref{t:main-es} using the sharpest equivalence constants, which are guaranteed to exist for finite-dimensional spaces. Now, we do not know if the equivalence constants of \cite[Theorem 5.2]{Aguilar-Garcia-Kim21} are sharp, but  we chose to present the proof of Theorem \ref{t:main-es} using these constants since they provided continuity with   explicit L-seminorms rather than L-seminorms that are built using unknown sharp constants. Thus, the purpose of Proposition \ref{p:sharp-cont} is to show that if one cannot  calculate explicit equivalence constants that provide continuity, then at least, one can achieve continuity with the existence of sharp equivalence constants. First, we prove a lemma.
 
 \begin{lemma}\label{l:min-conv}
 Let $(C, d)$ be a compact metric space. Let $(f_n)_{n \in \N}$ be a sequence of real-valued continuous functions on $X$, and let $f:X \rightarrow \R$ be continuous. 
 
 If $(f_n)_{n \in \N}$ converges to $f$ uniformly, then   $(\min_{x \in C} f_n(x))_{n \in \N}$ converges to $\min_{x \in C} f(x)$ and $(\max_{x \in C} f_n(x))_{n \in \N}$ converges to $\max_{x \in C} f(x)$.
 \end{lemma}
 \begin{proof}
 Since  $C$ is compact, $\inf_{x \in C} f(x)=\min_{x \in C} f(x)=\min f$ and $\inf_{x \in C} f_n(x)=\min_{x \in C} f_n(x)=\min f_n$ for all $n \in \N$.
Let $\varepsilon>0$. There exists an $N \in \mathbb{N}$ such that for $n\geq N$, we have $|f_n(x) -f(x)| < \varepsilon/2$ for all $x \in C$. Then for $n\geq N$,   $$f(x) - \varepsilon/2 < f_n(x) < f(x) +\varepsilon/2$$ for all $x \in C$. We take the infimum of this inequality to obtain
$$\min f - \varepsilon/2 \leq \min f_n \leq \min f +\varepsilon/2,$$ 
which implies $|\min f-\min f_n| \leq \varepsilon/2 < \varepsilon$.

 A similar argument establishes the result for $\min $ replaced with $\max$.
 \end{proof}
 
 \begin{proposition}\label{p:sharp-cont}
 Let $(V, \|\cdot\|)$ be a finite-dimensional normed vector space. Let \\  $(\|\cdot\|_n)_{n \in \N}$ be a sequence of norms on $V$ converging uniformly on the unit ball of  $(V, \|\cdot\|)$ to a norm $\|\cdot\|_\infty$ on $V$.  
 
 If for each $n \in \N \cup \{\infty \}$ we have 
 \[
 \kappa_n \|\cdot\| \leq \|\cdot\|_n \leq \lambda_n \|\cdot\|
 \] where $\kappa_n>0, \lambda_n >0$ are sharp, then 
 $(\kappa_n)_{n \in \N}$ converges to $\kappa_\infty$ and $ (\lambda_n)_{n \in \N}$ converges to $\lambda_\infty$.
 \end{proposition}
 \begin{proof}
 Let $n \in \N\cup \{\infty\}$. First, note that   $a \in V \mapsto \|a\|_n$ is continuous with respect to $\|\cdot\|$ since the norms are equivalent. Thus, since $\{a \in V:   \|a\|=1\}$ is compact by finite dimensionality,   $  \{\|a\|_n \in \R : a \in V, \|a\|=1\}$ is compact. Hence \[\inf  \{\|a\|_n \in \R : a \in V, \|a\|=1\}=\min  \{\|a\|_n \in \R : a \in V, \|a\|=1\}>0\] as $\|a\|=1$ implies that $a \neq 0$ and thus $\|a\|_n>0$.  Since $\kappa_n$ is sharp, 
 \[
 \kappa_n=\min  \{\|a\|_n \in \R : a \in V, \|a\|=1\}.
 \]
 Therefore, by Lemma \ref{l:min-conv},   $(\kappa_n)_{n \in \N}$ converges to $\kappa_\infty$ since the unit sphere of $(V, \|\cdot\|)$ is compact by finite dimensionality. The remaining result follows similarly.
 \end{proof}
 Thus, combining this result with Proposition \ref{p:cond-exp-conv}, we also have a proof of Theorem \ref{t:main-es} using the sharp constants for $\kappa^\theta_n$ in Expression \ref{eq:fd-es-lip} rather than the explicit ones of \cite[Theorem 5.2]{Aguilar-Garcia-Kim21}.
 
 For our final result, we present convergence of UHF algebras with respect to convergence of  their multiplicity sequences. Unlike the Effros--Shen case,  where the continuity result relied on continuity of the equivalence constants in some appropriate sense, convergence in  UHF algebras occurs regardless of which equivalence constants are chosen. First, we detail the metric space that we use to describe convergence of the multiplicity sequences and the standard construction of the class of UHF algebras.

\begin{definition}\label{Baire-Space-def}
The \emph{Baire space} $\BaireSpace$ is the set $(\N\setminus\{0\})^\N$ endowed with the metric $\mathsf{d}$ defined, for any two $(x(n))_{n\in\N}$, $(y(n))_{n\in\N}$ in $\BaireSpace$, by 
\begin{equation*}
 d_\BaireSpace\left((x(n))_{n\in\N}, (y(n))_{n\in\N}\right) = \begin{cases}
0  \ \ \ \ \text{ if $x(n) = y(n)$ for all $n\in\N$},\\ \\
2^{-\min\left\{ n \in \N : x(n) \not= y(n) \right\}} \ \ \  \text{ otherwise}\text{.}
\end{cases}
\end{equation*}
\end{definition}

Next, we define UHF algebras in a way that  suits  our needs. Given    $(\beta(n))_{n\in\N} \in \BaireSpace$, let 
\[
\boxtimes \beta(n)=\begin{cases}
1 & \text{ if } n=0,\\
\prod_{j=0}^{n-1} (\beta(j)+1) & \text{ otherwise}.
\end{cases}
\]

For each $n \in \N$, define a unital *-monomorphism by
\[
\mu_{\beta,n} : a \in \M_{\boxtimes \beta(n)}(\C) \longmapsto \mathrm{diag}(a,a,\ldots, a) \in \M_{\boxtimes \beta(n+1)}(\C),
\]
where there are $\beta(n)+1$ copies of $a$ in $\mathrm{diag}(a,a,\ldots,a)$. Set $\mathsf{uhf}((\beta(n))_{n\in\N})=\underrightarrow{\lim} \ (\M_{\boxtimes \beta(n)}(\C) , \mu_{\beta,n})_{n \in \N}$. The map
\[
(\beta(n))_{n\in\N} \in \BaireSpace \longmapsto \mathsf{uhf}((\beta(n))_{n\in\N})
\]
is a surjection onto the class of all UHF algebras up to *-isomorphism by \cite[Chapter III.5]{Davidson}.

For each $n \in \N$, let 
\[
\gamma_\beta(n)=\frac{1}{\dim(\M_{\boxtimes \beta(n)}(\C))},
\]
and let \[\rho_\beta\] be the unique faithful tracial state on $uhf((\beta(n))_{n\in\N})$, and set
\[
\mathcal{V}^\beta_n = \mu^{(n)}_\beta (\M_{\boxtimes \beta(n)}(\C)) 
\]
as in Expression \eqref{eq:ind-lim-maps}.

Next, let $\rho_{\beta,n+1}$ denote the unique faithful tracial state on $ \M_{\boxtimes \beta(n+1)}(\C)$. Fix $\lambda^\beta_n>0$ such that 
\begin{equation}\label{eq:free-equiv-const}
\lambda^\beta_n \|a\|_{\M_{\boxtimes \beta(n+1)}(\C)} \leq \|a\|_{E^{\rho_{\beta,n}}}
\end{equation}
for all $a \in  \M_{\boxtimes \beta(n+1)}(\C)$, 
where $E^{\rho_{\beta,n}}: \M_{\boxtimes \beta(n+1)}(\C) \rightarrow \mu_{\beta,n}(\M_{\boxtimes \beta(n)}(\C))$ is the unique faithful  $\rho_{\beta,n+1}$-preserving conditional expectation onto $\mu_{\beta,n}(\M_{\boxtimes \beta(n)}(\C))$. Here we note that $\lambda^\beta_n$ is neither explicit nor necessarily the sharp constant  and we assume that $\lambda^\beta_n$ only depends on $\M_{\boxtimes \beta(n+1)}(\C)$, which is allowed since   $\rho_{\beta,n+1}$ is the unique faithful tracial state on $\M_{\boxtimes \beta(n+1)}(\C)$. Let $\lambda^\beta=(\lambda^\beta_n)_{n \in \N}$.

 For each $a \in \mu^{(n)}_{\beta}(\M_{\boxtimes \beta(n)}(\C))$, let 
 \begin{equation}\label{eq:fd-uhf-lip}
 \Lip^\BaireSpace_{\beta,n}(a)=\Lip^{\rho_\beta, \lambda}_{\mathcal{V}^\beta_n,\gamma_\beta}(a)
 \end{equation}
 as in  Theorem \ref{t:main-af}.

\begin{theorem}\label{t:main-uhf}
The map
\[
\beta \in \BaireSpace \longmapsto (\mathsf{uhf}((\beta(n))_{n \in \N}),  \Lip^\BaireSpace_{\beta}),
\]
where $\Lip^\BaireSpace_{\beta}$ is defined in Theorem \ref{t:main-af} using the L-seminorms defined in Expression \eqref{eq:fd-uhf-lip}, 
is continuous with respect to the Baire space and the topology induced by $\dpropinquity{SL}$.
\end{theorem}
\begin{proof} The majority of this proof is complete by the proof of 
  \cite[Theorem 4.9]{Aguilar-Latremoliere15}. All that remains is continuity of the equivalence constants, but this follows similarly as the proof of \cite[Theorem 4.9]{Aguilar-Latremoliere15}.  Indeed,  if $d_\BaireSpace (\beta, \eta) <\frac{1}{2^n}$, then for all $k \leq n,$ we have  $\lambda^\beta_k=\lambda^\eta_k$ since 
\begin{enumerate}
\setlength\itemsep{0.4em}
\item     $\rho_{\beta,k+1}=\rho_{\eta,k+1}$, 
\item $ \M_{\boxtimes \beta(k+1)}(\C)= \M_{\boxtimes \eta(k+1)}(\C)$, 
\item $\mu_{\beta,k}(\M_{\boxtimes \beta(k)}(\C))=\mu_{\eta,k}(\M_{\boxtimes \eta(k)}(\C))$, and
\item $E^{\rho_{\beta,k}}=E^{\rho_{\eta,k}}$,
\end{enumerate}which is all the information required to fix $\lambda^\beta_k$ and $\lambda^\eta_k$. 
\end{proof}

\bibliographystyle{amsplain}
\bibliography{thesis}
\vfill

\end{document}